\documentclass[11pt,reqno]{amsart}
\usepackage{microtype}

\usepackage{mathtools}
\usepackage[dvipsnames]{xcolor}
\usepackage[colorlinks=true,linkcolor=Maroon,citecolor=OliveGreen]
{hyperref}
\usepackage{amssymb}
\usepackage{mathdots}
\usepackage[shortlabels]{enumitem}
\setlist[enumerate]{label={(\arabic*)}}

\usepackage{booktabs}

\usepackage[capitalize]{cleveref}
\crefname{equation}{}{}

\usepackage{thmtools}

\usepackage[abbrev]{amsrefs}  

\numberwithin{equation}{section}
\newtheorem{theorem}{Theorem}
\newtheorem{lemma}[theorem]{Lemma}

\newtheorem{proposition}[theorem]{Proposition}
\newtheorem{corollary}[theorem]{Corollary}

\theoremstyle{definition}

\newtheorem{question}[theorem]{Question}
\newtheorem*{acknowledgments}{Acknowledgments}

\def\R{\mathbf{R}} 
\def\Z{\mathbf{Z}} 
\def\F{\mathbf{F}} 

\def\G{\mathbf{G}}

\newcommand\opr[1]{\operatorname{#1}}
\def\Aut{\opr{Aut}}
\def\Sym{\opr{Sym}}

\def\GL{\opr{GL}}

\def\SL{\opr{SL}}
\def\SU{\opr{SU}}
\def\PSL{\opr{PSL}}
\def\PSU{\opr{PSU}}
\def\Sp{\opr{Sp}}
\def\PSp{\opr{PSp}}

\renewcommand\O{\opr{O}}

\def\Spin{\opr{Spin}}

\def\GU{\opr{GU}}
\def\Hom{\opr{Hom}}

\newcommand\BSW[1]{{\cite{BSW}*{#1}}}
\newcommand\MT[1]{{\cite{MT}*{#1}}}

\newcommand\fix{\opr{fix}}

\usepackage{todonotes}

\author{Jessica Anzanello}
\address{Dipartimento di Matematica e Applicazioni, University of Milano-Bicocca, Piazza del Calendario 3, 20125 Milano, Italy}
\email{j.anzanello@campus.unimib.it}

\author{Sean Eberhard}
\address{Mathematics Institute, Zeeman Building, University of Warwick, UK}
\email{sean.eberhard@warwick.ac.uk}

\thanks{%
This work was completed while J.A.\ was a visiting PhD student at the University of Warwick, and she gratefully acknowledges its hospitality.
S.E.\ is supported by the Royal Society through a University Research Fellowship (URF\textbackslash R1\textbackslash 221185).}

\begin{document}

\title
{Fixed-point-free elements in two-orbit permutation groups}
\begin{abstract}
    Let $G$ be a two-orbit permutation group on $n > 2$ points.
    We show that $G$ contains either a derangement or an element of prime-power order with a unique fixed point.
    As a corollary, if the orbits of $G$ have length $n_1$ and $n_2$ and $\gcd(n_1, n_2-1) = \gcd(n_1-1, n_2) = 1$, then $G$ contains a derangement.
    The special case $n_1 = n_2$ was recently conjectured by Ellis and Harper and proved under various restrictive hypotheses.
    We prove our result by reducing to the case of simple groups and leveraging the classification of normal $2$-coverings of simple groups due to Bubboloni, Spiga, and Weigel.
\end{abstract}

\maketitle

\section{Introduction}

If $G \le \Sym(\Omega)$ is a permutation group acting on a finite set $\Omega$, we write $\fix(g)$ for the number of fixed points on $\Omega$ of an element $g \in G$, and we say $g$ is a \emph{derangement} if $\fix(g) = 0$.

A classical theorem of Jordan states that every finite transitive group $G \le \Sym(\Omega)$ with $|\Omega| > 1$ contains at least one derangement. A deep theorem of Fein--Kantor--Schacher~\cite{FKS} adds that some derangement has prime-power order.

In this note we obtain the following generalization to two-orbit permutation groups.
The theorem of Fein--Kantor--Schacher coincides with the special case in which one of the orbits is trivial.

\begin{theorem}
    \label[theorem]{thm:main}
    Let $G \le \Sym(\Omega)$ be a two-orbit permutation group with $|\Omega| > 2$. Then $G$ contains either
    \begin{enumerate}
        \item a derangement, or
        \item an element of prime-power order with a unique fixed point.
    \end{enumerate}
\end{theorem}

Note that, in conclusion (1), we cannot require the derangement to have prime-power order. Indeed, $C_6$ has a two-orbit action of degree $5$, with point stabilizers $C_2$ and $C_3$, in which every derangement has order $6$.

Two-orbit permutation groups arise naturally in several contexts.
For instance, they occur in graph theory as automorphism groups of graphs that are edge-transitive but not vertex-transitive.
They also lead naturally to normal $2$-coverings of finite groups, whose motivations and applications are discussed in detail in the introduction of the book of Bubboloni--Spiga--Weigel \BSW{}.
For number-theoretic applications, see \cite{Muller}, \cite{EH}*{Remark 10}, and \cite{Saxl}.

Our original motivation for formulating \cref{thm:main} is a recent conjecture of Ellis and Harper~\cite{EH} (generalizing \cite{Saxl}*{Proposition~1}),
which predicts that every two-orbit permutation group with equal-sized nontrivial orbits contains a derangement.
Ellis and Harper proved their conjecture under various restrictive hypotheses, for example if $G$ is soluble, or almost simple, or primitive on at least one of the orbits.
Further results were obtained by Lee--Popiel--Verret~\cite{LPV}.

The Ellis--Harper conjecture follows easily from \Cref{thm:main},
and indeed the following corollary is more general.
From now on, we say that a two-orbit permutation group has \emph{orbit type} $(n_1, n_2)$ if its orbits have lengths $n_1$ and~$n_2$.

\begin{corollary}
    \label[corollary]{cor:main}
    Let $G \le \Sym(n)$ be a two-orbit permutation group of orbit type $(n_1, n_2)$, where $n = n_1 + n_2 > 2$.
    If \(\gcd(n_1, n_2-1) = \gcd(n_1-1, n_2) = 1\), then $G$ contains a derangement.
\end{corollary}
\begin{proof}
    Let $\Omega_1$ and $\Omega_2$ be the orbits, where $|\Omega_i| = n_i$.
    Suppose $g \in G$ is a $p$-element with a unique fixed point $x$.
    If, say, $x \in \Omega_1$ then $p$ divides $n_1-1$ and $n_2$,
    and similarly if $x \in \Omega_2$.
    Therefore if $\gcd(n_1, n_2 - 1) = \gcd(n_1 - 1, n_2) = 1$ then $G$ cannot contain an element of prime-power order with a unique fixed point, so \Cref{thm:main} implies that $G$ contains a derangement.
\end{proof}

We record the following graph-theoretic consequence of \Cref{cor:main}, which was already observed as a consequence of the Ellis--Harper conjecture in \cite{EH}*{Remark~8}.

\begin{corollary}
    \label{cor:edge-transitive-graphs}
    Let $\Gamma$ be a finite edge-transitive regular graph of degree $n > 1$.
    Then $\Aut(\Gamma)$ contains a derangement in its action on the vertices of $\Gamma$.
\end{corollary}

\begin{proof}
    If $\Gamma$ is vertex-transitive then the result follows from Jordan's theorem.
    Otherwise, since $\Gamma$ is regular and edge-transitive, $\Aut(\Gamma)$ has exactly two orbits on the vertices, and $\Gamma$ is bipartite.
    Moreover, regularity implies that these two orbits have the same size. The result therefore follows from \Cref{cor:main}.
\end{proof}

The coprimality conditions in \Cref{cor:main} are sufficient
but not necessary for the existence of a derangement in a two-orbit
permutation group.
For instance, every two-orbit permutation group of orbit type $(3, 55)$ contains a derangement, although this
pair does not satisfy the hypotheses of \Cref{cor:main}. This suggests the following natural question.

\begin{question}
    Which pairs $(n_1,n_2)$ of integers have the property that every two-orbit permutation group of orbit type $(n_1, n_2)$ contains a derangement?
\end{question}

The proof of \Cref{thm:main} is easy to outline. A key advantage of the formulation of \Cref{thm:main}, compared to the Ellis--Harper conjecture, is that it reduces easily to the case of simple groups.
This reduction is given in \Cref{sec:reduction}.
In the case of simple groups, we use the classification of normal $2$-coverings of finite simple groups due to Bubboloni--Spiga--Weigel \BSW{}.
There are several infinite families of classical and exceptional groups, mostly with a similar structure,
as well as a handful of cases that we check with a computer.

Fein--Kantor--Schacher called it ``unfortunate and perhaps even outrageous'' that their proof depends on detailed knowledge of the structure of the finite simple groups.
Forty years on, still no proof avoiding the classification is known.
Our result is a new contribution to this unfortunate, and indeed perhaps even outrageous, state of affairs.

\begin{acknowledgments}
The authors are grateful to Scott Harper, Luca Sabatini, Pablo Spiga and Adam Thomas for useful conversations.
\end{acknowledgments}

\section{Reduction to simple groups}
\label{sec:reduction}

In this section, we reduce the proof of \cref{thm:main} to the case of simple groups.
We begin by recalling a well-known generalization of Burnside's orbit-counting lemma (cf.~\cite{EH}*{Lemma~2.3}).

\begin{lemma}
    \label{lemma:burnside}
    Let $H \le \Sym(\Omega)$ be transitive and let $g \in \Sym(\Omega)$. Then the average number of fixed points of elements in the coset $Hg$ is 1.
\end{lemma}
\begin{proof}
    Since $H$ is transitive,
    \begin{align*}
        \sum_{h \in H}\fix(hg)
        &= |\{(h,\omega) \in H \times \Omega \mid \omega^{hg} = \omega\}|\\
        &= \sum_{\omega \in \Omega} |\{h \in H \mid \omega^h = \omega^{g^{-1}} \}|
        = \sum_{\omega \in \Omega} |H_{\omega}| = |H|.\qedhere
    \end{align*}
\end{proof}

\begin{proposition}
    \label{prop:reduction-simple}
    A minimal counterexample to \Cref{thm:main} must be simple.
\end{proposition}

\begin{proof}
Suppose $G \le \Sym(\Omega)$ is a counterexample to \Cref{thm:main} of minimal order, i.e.,
\begin{enumerate}
    \item $G \le \Sym(\Omega)$ has $2$ orbits and $|\Omega| > 2$,
    \item $\fix(g) \ge 1$ for all $g \in G$,
    \item $\fix(g) \ge 2$ for all $g \in G$ of prime-power order,
\end{enumerate}
and $|G|$ is minimal with respect to these conditions.
We claim that $G$ must be simple.

If not, $G$ has a proper nontrivial normal subgroup $N \lhd G$.
By minimality of $G$, $N$ must have at least $3$ orbits on $\Omega$.
Let $\Delta = \Omega / N$ be the partition of $\Omega$ into $N$-orbits.
Then $G/N$ acts on $\Delta$ with two orbits and $|\Delta| > 2$.
Let $G^*$ be the image of $G / N$ in $\Sym(\Delta)$.
Then $|G^*| \le |G / N| < |G|$, so, by minimality of $G$, either (2) or (3) fails for $G^*$.

If (2) fails then there is an element $g \in G$ that has no fixed $N$-orbit, which is impossible since $\fix(g) \ge 1$ for all $g \in G$.

If (3) fails then for some prime $p$ there is a $p$-element $g^* \in G^* \le \Sym(\Delta)$ such that $\fix(g^*) = 1$,
and we may lift $g^*$ to a $p$-element $g \in G$.
Let $\Sigma \in \Delta$ be the unique fixed point of $g^*$.
Then, for every $h \in Ng$, $h$ fixes $\Sigma$ and no other $N$-orbit,
so in particular $\fix(h) = \fix(h|_\Sigma)$ where $h|_\Sigma$ denotes the restriction of $h$ to $\Sigma$.
However, since $N$ is transitive on $\Sigma$, by \Cref{lemma:burnside} we have
\[
    \frac1{|N|} \sum_{h \in Ng} \fix(h|_\Sigma) = 1.
\]
Since $\fix(g) \ge 1$ for all $g \in G$, it follows that every $h \in Ng$ has a unique fixed point, but this is impossible because $\fix(g) \ge 2$.
\end{proof}

\section{Simple groups}
By \Cref{prop:reduction-simple}, it remains to prove \Cref{thm:main} under the assumption that $G$ is simple.
Thus let $G \le \Sym(\Omega)$ be a simple two-orbit permutation group with $|\Omega| > 2$, and suppose that $G$ contains no derangements.
If $G \cong C_p$ then every nontrivial element of $G$ is a $p$-element with at most one fixed point,
so we may assume $G$ is a nonabelian simple group.
If one of the orbits is a point then we refer to the theorem of Fein--Kantor--Schacher,
so we may assume both orbits are nontrivial.
Let $H, K < G$ be the point stabilizers.
Then $\{H, K\}$ is a normal $2$-covering of $G$, i.e.,
\[
    G = \bigcup_{g \in G} H^g \cup \bigcup_{g \in G} K^g,
\]
and thus we may refer to the classification of normal $2$-coverings of nonabelian simple groups due to Bubboloni--Spiga--Weigel: see~\BSW{Theorem~11.1}.
We need to inspect \BSW{Tables~11.1--8}, and in each case we need to check that some element of prime-power order is contained in exactly one conjugate of $H$ and no conjugate of $K$, or vice versa.

\subsection{Special cases}
\label{smallcases}

The smallest cases can be discarded on the simple grounds that in fact every element in the group has prime-power order, together with the elementary observation that some element $g \in G$ has at most one fixed point.
Indeed, by \Cref{lemma:burnside} applied to each of the two orbits, we have
\[
    \frac1{|G|}\sum_{g \in G} \fix(g) = 2,
\]
and $\fix(1) = n > 2$, so $\fix(g) < 2$ for some $g \in G$.
We can thus discard the cases
\[
A_5,\ A_6,\ \PSL_2(7),\ \PSL_2(8),\ \PSL_3(4).
\]

We use \texttt{GAP} to deal with several further sporadic and small-rank cases, namely the groups
\begin{center}
     $A_7,\ A_8,\ M_{11},\ {}^2F_4(2)',
     \ \PSU_3(3),\ \PSU_3(5),$
    $\PSU_4(2)\cong \PSp_4(3),\ \PSU_4(3).$ 
\end{center}
For each of these groups, we inspect the primitive actions arising from the normal $2$-coverings listed in \BSW{Tables~11.2--11.8}.
For each pair of primitive actions, and for each prime-power order $r$ arising among elements of $G$, we record the maximum number of fixed points among elements of order $r$ in both actions.
In every case there is some prime power $r$ such that, for every element $g \in G$ of order $r$, the total number of fixed points in the two actions is $1$.
In most but not all cases the fixed point can be taken to be in either orbit.
The output is shown in \Cref{tab:gap-small-cases}.

\begin{table}[t]
    \centering
    \footnotesize
    \caption{\footnotesize Some special cases checked with \texttt{GAP}. For each pair of primitive actions \((\Omega_1,\Omega_2)\) arising from the normal \(2\)-coverings in \mbox{\BSW{Tables~11.2--11.8}}, and for either index $i \in \{1, 2\}$, the table lists, if possible, a prime-power order \(r = r_i\) such that $\fix_{\Omega_1\cup\Omega_2}(g) = 1$ and $\fix_{\Omega_i}(g) = 0$ for every element $g \in G$ of order $r$.}
    \label{tab:gap-small-cases}
    \[
    \begin{array}{llllrr}
        \toprule
        G & H & K & \text{orbit type} & r_1 & r_2 \\
        \midrule
        A_7 & A_7 \cap (S_2 \times S_5) & \SL_3(2) & (21, 15) & 7 & 5 \\
        A_8 & A_8 \cap (S_3 \times S_5)&2^3:\SL_3(2)& (56, 15) & 7 & 5 \\
        M_{11} & M_{10} & \PSL_2(11) & (11, 12) & 11 & 2^3 \\
        & M_9:2 & \PSL_2(11) & (55, 12) & 11 & 2^3 \\
        & M_8:S_3 & \PSL_2(11) & (165, 12) & 11 & 2^2 \\
        {}^2F_4(2)' & 2\cdot [2^8]:5:4 & \PSL_3(3):2 & (1755,1600) & 13 & 2^4\\
        & 2 \cdot [2^8]:5:4 & \PSL_2(25) & (1755, 2304) & - & 2^4 \\
        \PSU_3(3) & \PSL_2(7) & 3^{1+2} :8 & (36,28) & - & 7 \\
        & \PSL_2(7) & \GU_2(3) & (36, 63) & 2^3 & 7 \\
        \PSU_3(5) & A_7 & \frac{1}{3} 5^{1+2}:24 & (50, 126) & - & 7 \\
        \PSU_4(2) & \GU_3(2) & 2^4 :\SL_2(4) & (40,27) & - & 3^2 \\
        & \GU_3(2) & \Sp_4(2) & (40,36) & 5 & 3^2 \\
        \PSU_4(3) & \frac{1}{4}\GU_3(3) & \frac{1}{4} 3^4:\SL_2(9):2 & (540,112) & 3^2 & 7 \\
        & \PSL_3(4) & \frac{1}{4} 3^{1+4}:\SU_2(3):8 & (162, 280) & 3^2 & 7 \\
        &A_7 & \frac{1}{4} 3^{1+4}:\SU_2(3):8 & (1296,280) & 3^2 & 5 \\
        \bottomrule
    \end{array}
    \]
\end{table}

In some of the cases appearing in \BSW{Tables~11.1--11.8}, one of the two actions is imprimitive, and thus not directly covered by \Cref{tab:gap-small-cases}.
These cases can be checked indirectly from \Cref{tab:gap-small-cases} as follows.
Suppose the orbits are $(\Delta_1, \Omega_2)$ where $\Delta_1$ is imprimitive and $\Omega_2$ is primitive.
Let $\Omega_1$ be a minimal primitivity system in $\Delta_1$.
Then $(\Omega_1, \Omega_2)$ or $(\Omega_2, \Omega_1)$ appears in \Cref{tab:gap-small-cases}.
In each such case, there happens to be an element $g \in G$ of prime-power order with no fixed points in $\Omega_1$ and a unique fixed point in $\Omega_2$.
It follows that $g$ has the same properties with respect to $(\Delta_1, \Omega_2)$.
The only case in which both actions may be imprimitive is $G = \PSL_3(4)$ (as noted in \BSW{Corollary~11.2}), which has already been excluded by the preceding argument.

\subsection{Classical groups}
\label{sec:classical}

Next, we consider the infinite families of classical groups appearing in \BSW{Tables~11.4--11.8}. For convenience, we replace each simple group by its natural quasisimple cover acting faithfully on the corresponding linear, unitary, or symplectic space $V$.

Recall that a \emph{primitive prime divisor} of $q^n - 1$ is a prime $r$ such that $r \mid q^n - 1$ but $r \nmid q^i - 1$ for $0 < i < n$.
By a celebrated theorem of Zsigmondy~\cite{Zsigmondy}, such a prime exists unless $n = 2$ and $q$ is a Mersenne prime or $(n, q) = (6, 2)$.

\subsubsection{Linear groups}

Refer to \BSW{Tables~11.4--6}.

\emph{Case $G = \SL_2(q)$.}
Here $H$ is either a nonsplit torus $T \cong C_{q+1}$ or its normalizer $N_G(T)$, while $K$ is contained in either a Borel subgroup or the normalizer of a split torus, so $|K|$ divides $2q(q-1)$.
If $q$ is not a Mersenne prime, let $r$ be a primitive prime divisor of $q^2 - 1$;
otherwise let $r = q + 1$.
Let $s \in T$ be an element of order $r$.
Then $T = C_G(s)$, so $s$ is contained in a unique conjugate of $T$ and $N_G(T)$, and $s$ is contained in no conjugate of $K$, as $r$ does not divide $2q(q-1)$. Alternatively, if $q$ is a Mersenne prime, one may instead use a unipotent element of order $q$.

\emph{Case $G = \SL_3(q)$.}
The case $\SL_3(4)$ was excluded in the previous section,
and in all other cases $H$ is either a torus $T$ of order $q^2+q+1$ or its normalizer, while $K$ is either a parabolic subgroup $q^2 : \GL_2(q)$ or a Levi complement $\GL_2(q)$.
Let $r$ be a primitive prime divisor of $q^3 - 1$ and let $s \in T$ be an element of order $r$.
The argument of the previous case applies again, as $r$ does not divide $|q^2:\GL_2(q)| = q^3(q^2 -1)(q-1)$.

\emph{Case $G = \SL_4(q)$.}
Here $H \le \SL_2(q^2) \cdot (q+1) \cdot 2$ while $K \cong q^3 : \GL_3(q)$ is the stabilizer of a one- or three-dimensional subspace.
Let $r$ be a primitive prime divisor of $q^3 - 1$, and let $s \in G$ be an element of order $r$.
Then $s$ is reducible of type $3 + 1$, so $s$ is contained in a unique conjugate of $K$,
and $s$ is contained in no conjugate of $H$, as $r$ does not divide $|\SL_2(q^2) \cdot (q+1) \cdot 2| = 2 q^2 (q^4 - 1) (q+1)$.

\subsubsection{Unitary groups}

Refer to \BSW{Table~11.7}. The module is $V = \F_{q^2}^n$.

\emph{Case $G = \SU_3(q)$, $q = 3^f > 3$.}
Here $H$ is the normalizer of a torus $T \cong C_{q^2 - q + 1}$, and $K \cong \GU_2(q)$ is the stabilizer of a nondegenerate line.
It suffices to let $r$ be a primitive prime divisor of $q^6 - 1$, and let $s \in T$ be an element of order $r$.
Then $T=C_G(s)$, so $s$ is contained in a unique conjugate of $H=N_G(T)$, and $s$ is contained in no conjugate of $K$ as $r$ does not divide $|K| = q (q^2 - 1) (q + 1)$.

\emph{Case $G = \SU_4(q)$, $q > 3$.}
Here $H \cong \GU_3(q)$ is the stabilizer of a decomposition $V = U \oplus U^\perp$ with $\dim U = 3$, and $K \cong q^4 : \SL_2(q^2)  : (q-1)$.
Let $r$ be a primitive prime divisor of $q^6 - 1$, and let $s \in G$ be an element of order $r$.
Then $s$ is reducible of type $3 + 1$, so $s$ is contained in a unique conjugate of $H$,
and $s$ is contained in no conjugate of $K$, as $r$ does not divide $|K| = q^6 (q^4 - 1)(q-1)$.

\subsubsection{Symplectic groups}

Refer to \BSW{Table~11.8}.

\emph{Case $G = \Sp_6(q)$, $q=3^f$.}
Here $H \cong \Sp_2(q) \times \Sp_4(q)$ is the stabilizer of a decomposition $V = U \oplus U^\perp$ with $\dim U = 2$, and $K \cong \Sp_2(q^3) : 3$.
There is a maximal torus $T \le H$ such that $T \cong C_{q+1} \times C_{q^2 + 1}$.
Let $r$ be a primitive prime divisor of $q^4 - 1$, and let $s \in T$ be an element of order $r$.
Then $s$ is reducible of type $1+1+4$, so $s$ is contained in a unique conjugate of $H$,
and $s$ is not contained in any conjugate of $K$, as $r$ does not divide $|K| = 3 q^3 (q^6 - 1)$.

\emph{Case $G = \Sp_{2n}(q)$, $q=2^f$, $n \ge 2$.}
Here $H \cong \O_{2n}^-(q)$ and $K \cong \O_{2n}^+(q)$.
We may take $H$ to be the group of isometries of a fixed quadratic form $f : V \to \F_q$ of minus type whose polar form is the $G$-invariant sympletic form.
There is a maximal torus $T \le H$ such that $T \cong C_{q^n + 1}$.
Deferring the case $(n, q) = (3, 2)$, let $r$ be a primitive prime divisor of $q^{2n} - 1$ and let $s \in T$ be an element of order $r$.
Then $s$ is irreducible.
We claim this implies that $s$ is contained in a unique conjugate of $H$.
Indeed, it suffices to show that there is a unique $s$-invariant quadratic form $f : V \to \F_q$ whose polar form is the given symplectic form.
Let $f_1, f_2$ be two such forms and let $\delta = f_1 - f_2$.
Then $\delta$ is an $s$-invariant quadratic form with zero polar form, which implies that $\delta(v) = \ell(v)^2$ for some $s$-invariant linear form $\ell : V \to \F_q$.
Let $W = \{v \in V: \ell(v) = 0\}$.
Then $W$ is an $s$-invariant subspace of codimension at most $1$.
Since $s$ is irreducible, it follows that $W = V$ and $f_1 = f_2$, proving the claim.
On the other hand, $s$ is contained in no conjugate of $K$ as $r$ does not divide $|K| = 2 q^{n(n-1)}(q^n-1) \prod_{i=1}^{n-1} (q^{2i} - 1)$.

This leaves the case $(n, q) = (3, 2)$, as there is no primitive prime divisor of $2^6 - 1$.
In this case $T \cong C_9$, so let $s \in T$ be an element of order $9$.
The same argument shows that $s$ is contained in a unique conjugate of $H$,
and $s$ is contained in no conjugate of $K$ because $K = \O_6^+(2) \cong S_8$ has no element of order $9$.
Alternatively, in this case, one may instead use an element of order $7$.

\subsection{Exceptional groups}
\label{sec:exceptional}

There are two infinite families in \BSW{Table~11.3} of exceptional groups, specifically of types $G_2$ and $F_4$.
In these cases we argue via the Weyl groups of the ambient algebraic groups.
We remark that the classical cases could also have been analyzed in this way.

We briefly recall some background, largely following Malle--Testerman~\MT{}.
Let \(k\) be the algebraic closure of a finite field of characteristic \(p\).
Let \(G\) be a simple linear algebraic group over \(k\) with Steinberg endomorphism \(F:G\to G\).
The corresponding finite group of Lie type is the subgroup \(G^F \le G\) of $F$-fixed points. Let \(T\) be an \(F\)-stable maximal torus of \(G\).
The corresponding \emph{Weyl group} and \emph{character group} are
\[
    W = W_G(T) = N_G(T)/T,\qquad
    X = X(T) = \Hom(T,\G_m).
\]
Since \(T\) is \(F\)-stable, \(F\) induces an automorphism of \(W\), and by pullback it also induces a compatible endomorphism of \(X\). On \(X_\R=X\otimes_\Z\R\), this endomorphism has the form \(q\phi\), where \(q\) is a positive fractional power of \(p\) and \(\phi\) is a finite-order automorphism preserving the root system $\Phi \subset X_\R$. The scalar \(q\) has no effect on the induced action on \(W\), so we may use the same symbol \(\phi\) for the automorphism of \(W\) induced by \(F\); see \MT{22.1--22.2}.
Except for the very twisted cases \({}^2B_2,{}^2G_2,{}^2F_4\), \(q\) is an integral power of \(p\) and the same parameter appearing in the notation for the finite group of Lie type.

We shall make use of the fact that $\phi \in \Aut(X_\R)$ mildly depends on the choice of $T$.
By \MT{25.1, 21.9}, there is a bijection
\[
    \{G^F\text{-classes of \(F\)-stable maximal tori in}~G\}
    \longleftrightarrow \{W\text{-classes in}~W\phi\}.
\]
Here \(W\phi\) denotes the coset of \(\phi\) in the semidirect product \(W\rtimes\langle\phi\rangle\), where \(\phi\) acts on \(W\) by \(w^\phi=\phi(w)\).
If a different $F$-stable maximal torus $T^g$ is chosen instead of $T$ then the character groups $X(T)$ and $X(T^g)$ are related by conjugation by $g$,
and the new $\phi$ is given by the corresponding element of \(W\phi\).
Thus we have the liberty to modify $\phi \in \Aut(X_\R)$ by any element of $W$ \MT{p.~195}.

Refer to \BSW{Table~11.3}.

\subsubsection{Type $G_2$}

Here the relevant finite group is $G_2(q)$, $q = 2^f \ge 4$, with components of the form $\SL_3(q) \cdot 2$ and $\SU_3(q) \cdot 2$.
Write $G_2(q) = G^F$ where $G = G_2(k)$, $k$ is the algebraic closure of $\F_2$, and $F = F_q$ is the standard Frobenius endomorphism.
Then the components have the form $H^F$ and $K^F$, where $H$ and $K$ are both normalizers of suitable $F$-stable subsystem subgroups of type $A_2$.
The Weyl group of $G$ is $W_G \cong D_{12}$, while $W_{H^\circ} \cong D_6$, so $|H:H^\circ| = 2$ and $W_H = W_G$ by \MT{13.8}; this implies $N_G(T) \le H$ for every maximal torus $T \le H$.
There is an $F$-stable maximal torus $T \le H$ such that $\phi|_{X(T)}$ has order $3$ and $T^F \cong C_{q^2+q+1}$.
Let $r$ be a primitive prime divisor of $q^3 - 1 = (q-1) (q^2 + q + 1)$ and let $s \in T^F$ be an element of order $r$.
We claim that $s$ is contained in a unique $G^F$-conjugate of $H^F$.
Indeed, suppose $s \in (H^F)^g$, where $g \in G^F$.
Then $s^{g^{-1}} \in H^\circ$ since $|H:H^\circ| = 2$ and $r \ne 2$.
Since $s$ is regular semisimple it follows that $T^{g^{-1}} \le H^\circ$.
By conjugacy of maximal tori in $H$ it follows that $T^{g^{-1}} = T^h$ for some $h \in H$, so $g \in H N_G(T) \le H$.
Hence $g \in H^F$ and $(H^F)^g = H^F$.
Moreover, $s$ is contained in no conjugate of $K^F$ as $r$ does not divide $|K^F| = |\SU_3(q) \cdot 2| = 2 q^3 (q^3 + 1) (q^2 - 1)$.

\subsubsection{Type $F_4$}

Here the relevant finite group is $F_4(q)$, $q = 3^f$, with components of the form ${}^3D_4(q) \cdot 3$ and $\Spin_9(q)$.
Write $F_4(q) = G^F$ where $G = F_4(k)$, $k$ is the algebraic closure of $\F_3$, and $F = F_q$ is the standard Frobenius endomorphism.
Then the components have the form $H^F$ and $K^F$, where $H$ and $K$ are normalizers of suitable $F$-stable subsystem subgroups of type $D_4$ and $B_4$, respectively.
The Weyl group of $G$ is $W_G = W_{F_4} = W_{D_4} : S_3$.
This implies that $H = H^\circ \cdot S_3$ and $W_H = W_G$ by \MT{13.8}, and hence $N_G(T) \le H$ for every maximal torus $T \le H$.
There is an $F$-stable maximal torus $T \le H$ such that $\phi|_{X(T)}$ has order $12$ (i.e., corresponds to a Coxeter element of $W_{F_4}$) and $T^F \cong C_{q^4 - q^2 + 1}$.
Let $r$ be a primitive prime divisor of $q^{12} - 1$, so $r \mid q^4 - q^2 + 1$, and let $s \in T^F$ be an element of order $r$.
Arguing exactly as in the previous case, $s$ is contained in a unique $G^F$-conjugate of $H^F$, and $s$ is contained in no conjugate of $K^F$ as $r$ does not divide $|K^F| = |\Spin_9(q)| = q^{16} (q^8 - 1) (q^6 - 1) (q^4 - 1) (q^2 - 1)$.

This concludes the proof of \Cref{thm:main}.

 As a final remark, the proof shows that, in almost all cases, if $G$ is a simple two-orbit permutation group with no derangements, then $G$ contains an element of \textit{prime} order fixing exactly one point. Indeed, when both orbits are nontrivial, this holds for all the infinite families considered above. Among the small cases checked in \cref{smallcases} and \cref{tab:gap-small-cases}, one can check with \texttt{GAP} that the only exceptions are
$G={}^2F_4(2)'$ with point stablizers $H \cong 2 \cdot [2^8]:5:4$ and $K \cong \PSL_2(25)$, and $G=\PSU_4(2)$ with point stabilizers $H\cong\GU_3(2)$ and  $K \cong 2^4:\SL_2(4)$.
If one orbit is trivial, the only exception is represented by $M_{11}$ with point stabilizer $\PSL_2(11)$, as proved by Giudici~\cite{Giudici}.

\bibliography{refs}
\end{document}